\documentclass[a4paper,11pt]{amsart}
\pdfoutput=1 
\usepackage{amsmath,amsthm,amssymb}
\usepackage[mathscr]{eucal}
 \usepackage{cite}
\usepackage{upgreek}
\usepackage[bookmarks=false]{hyperref}
\usepackage{enumerate}
\usepackage{amsmath}
\usepackage{mathrsfs}
\usepackage{blindtext}
\usepackage{scrextend}
\usepackage{enumitem}
\usepackage{bm} 
\addtokomafont{labelinglabel}{\sffamily}
\usepackage{color}
\usepackage[at]{easylist}

\usepackage{comment}

\numberwithin{equation}{section}

\theoremstyle{plain}
\newtheorem{main}{Theorem}
\newtheorem{mcor}[main]{Corollary}

\newtheorem{theorem}{Theorem}[section]
\newtheorem{claim}[theorem]{Claim}
\newtheorem{lemma}[theorem]{Lemma}
\newtheorem{proposition}[theorem]{Proposition}
\newtheorem{corollary}[theorem]{Corollary}
\theoremstyle{definition}
\newtheorem{definition}[theorem]{Definition}
\newtheorem*{definition*}{Definition}

\newtheorem{remark}[theorem]{Remark}
\newtheorem{fact}[theorem]{Fact}

\begin{document}

\title[An exotic II$_1$ factor without property Gamma]
{An exotic II$_1$ factor without property Gamma}

\author[Ionu\c t Chifan]{Ionu\c t Chifan}
\address{Department of Mathematics, The University of Iowa, 14 MacLean Hall,  IA, 52242, U.S.A.}\email{ionut-chifan@uiowa.edu}
\urladdr{https://sites.google.com/view/i-chifan-website/}

\author[Adrian Ioana]{Adrian Ioana}
\address{Department of Mathematics, University of California San Diego, 9500 Gilman Drive, La Jolla, CA 92093, USA}\email{aioana@ucsd.edu}
\urladdr{https://mathweb.ucsd.edu/~aioana/}

\author[Srivatsav Kunnawalkam Elayavalli]{Srivatsav Kunnawalkam Elayavalli}
\address{Institute of Pure and Applied Mathematics, UCLA, 460 Portola Plaza, Los Angeles, CA 90095, USA}\email{srivatsav.kunnawalkam.elayavalli@vanderbilt.edu}
\urladdr{https://sites.google.com/view/srivatsavke/home}

{\thanks{I.C. was partially supported by NSF FRG Grant DMS-1854194 and NSF Grant DMS-2154637; A.I. was partially supported by NSF FRG Grant \#1854074, NSF Grant \#DMS-2153805 and a Simons Fellowship; S.K.E was supported by a Simons Postdoctoral Fellowship. }}

\begin{abstract} 
We introduce a new iterative amalgamated free product construction of II$_1$ factors, and use it to construct a separable II$_1$ factor which does not have property Gamma and is not elementarily equivalent to the free group factor $\text{L}(\mathbb F_n)$, for any $2\leq n\leq \infty$. This provides the first explicit example of two non-elementarily equivalent II$_1$ factors without property Gamma. Moreover, our construction also provides the first explicit example of a  II$_1$ factor without property Gamma that is also not elementarily equivalent to any ultraproduct of matrix algebras. Our proofs use a blend of techniques from Voiculescu's free entropy theory and Popa's deformation/rigidity theory. 

 \end{abstract}

\maketitle

\section{Introduction}

The   study of  the continuous model theory of II$_1$ factors  was initiated by Farah, Hart and Sherman in \cite{FHS}, who adapted the notion of elementary equivalence (requiring that the objects considered satisfy the same first-order sentences) to  the context of II$_1$ factors.
By the continuous version of the  Keisler-Shelah theorem, two II$_1$ factors $M,N$ are {\it elementarily equivalent} if and only if they admit isomorphic ultrapowers, $M^\mathcal U\cong N^\mathcal V$, for some ultrafilters $\mathcal U,\mathcal V$ on arbitrary sets \cite{FHS,hensoniovino}. Ultrapowers of II$_1$ factors have been a major tool in operator algebras since the works of McDuff \cite{McD2} and Connes \cite{Connes} in the 1970s. 
In spite of this, proving that two given II$_1$ factors have no isomorphic ultrapowers, and so are not elementarily equivalent, remains a challenging task. 

As shown in \cite{FHS} (see also \cite{FGL}), for separable II$_1$ factors,  Murray and von Neumann's property Gamma \cite{MvN43} and McDuff's property \cite{McD2}  are axiomatizable and thus are remembered by ultrapowers.  This implies that the hyperfinite II$_1$ factor $R$, the free group factor $\text{L}(\mathbb F_2)$ and any separable non-McDuff II$_1$ factor with property Gamma (see \cite{DiLa}) are not elementarily equivalent. It was then realized by Goldbring and Hart that a II$_1$ factor introduced in \cite{ZM} provides a fourth elementary equivalence class (see \cite{GH}).
However, besides these examples, it was unclear how to find any additional elementary equivalence classes of II$_1$ factors. 
This problem was solved by Boutonnet and two of the authors in \cite{BCI15} who proved that the continuum of non-isomorphic separable II$_1$ factors $(M_\alpha)_{\alpha\in\{0,1\}^\mathbb N}$ constructed by McDuff in \cite{McD1} are pairwise non elementarily equivalent. More precisely, the main result of \cite{BCI15} shows that $M_\alpha$ and $M_\beta$ do not admit isomorphic ultrapowers, whenever $\alpha\not=\beta$. Subsequently, explicit  sentences witnessing that $M_\alpha$ and $M_\beta$ are not elementarily equivalent were given in \cite{GH,GHT}. 

The proofs of the main result of \cite{BCI15} and in fact of all of the existing results providing non-elementarily equivalent II$_1$ factors are based on analyzing central sequences. 
As a result, it remained a fundamental open question to find any non-elementarily equivalent II$_1$ factors that do not have property Gamma and thus admit no non-trivial central sequences.

We settle this question in the present work. A main novelty of our approach, that allows us to bypass the above difficulty, is  the use of $1$-bounded entropy from Voiculescu's free probability theory. For a finite tuple $X$ of self-adjoint operators in a tracial von Neumann algebra $(N,\tau)$, one has the {\it $1$-bounded entropy} $h(X)$, implicit in Jung's work \cite{Jung2007} and defined explicitly by Hayes \cite{Hayes2018}, see Subsection \ref{1bdd}. This quantity, unlike  Voiculescu' free entropy dimension $\delta_0(X)$ \cite{VoiculescuFreeEntropy2},  is known to be an invariant of the von Neumann algebra generated by $X$ as shown in \cite[Theorem A.9]{Hayes2018}. Hence, we have a well-defined notion of $1$-bounded entropy $h(N)$ for a finitely generated tracial von Neumann algebra $(N,\tau)$. Moreover, $h(N)$ extends to arbitrary, possibly non-separable, tracial von Neumann algebras $(N,\tau)$ by \cite[Definition A.2]{Hayes2018}. 

The  main result of this paper  is the following:

\begin{main} \label{corollary about non ee of free group factors with our thing}
 There exists a separable II$_1$ factor $M$ which does not have property Gamma and is not elementarily equivalent to 
  any  tracial von Neumann algebra $(N,\tau)$ satisfying $h(N)>0$.    For instance, $M$ is not elementarily equivalent to $\emph{L}(\mathbb{F}_2)$.
  
  Moreover,  for any ultrafilters $\mathcal U,\mathcal V$ on sets $I,J$, respectively, there does not exist an embedding of $M^{\mathcal U}$ into $N^{\mathcal V}$ that contains the diagonal inclusion of $N$.

\end{main}

Examples of tracial von Neumann algebras $(N,\tau)$ with $h(N)>0$ include the interpolated free group factors
$\text{L}(\mathbb F_t)$, for all $1<t\leq \infty$, and, more generally, any free product $N_1*N_2$ of two  Connes-embeddable diffuse tracial von Neumann algebras $(N_1,\tau_1)$ and $(N_2,\tau_2)$. (Moreover, $h(N)=\infty$ for such $N$; for this and additional examples, see Fact \ref{examples of h>0}). By Theorem \ref{corollary about non ee of free group factors with our thing},  $M$ is not elementarily equivalent to any such $N$, including $\text{L}(\mathbb F_2)$.
This gives the first explicit example of two non-elementarily equivalent non-Gamma II$_1$ factors, thus settling a problem posed at a 2018 workshop at the American Institute of Mathematics \cite[Problem 1.3]{AIM}, see also \cite{IP} and \cite[Problem U.2]{JessePL}. 

It has been speculated for some time that free probability theory is likely to shed light on the model-theoretic study of II$_1$ factors, see for instance Farah's ICM survey \cite[Section 5]{FarahICM} and  \cite{BIRS}. Offering positive evidence in this direction, Theorem \ref{corollary about non ee of free group factors with our thing} represents the first application of free probability to the model theory of II$_1$ factors.

Now we describe the key facets of our construction that allows  us to prove Theorem \ref{corollary about non ee of free group factors with our thing}.  The II$_1$ factor from Theorem \ref{corollary about non ee of free group factors with our thing} is built via a new iterative construction involving amalgamated free products (see Section \ref{construction}).    
By using techniques from Popa's deformation/rigidity theory, notably \cite{IPP}, and the notion of property (T), we are able to guarantee that $M$ is indeed non-Gamma. The main property  of our construction is presented in our second main theorem below.

\begin{main}\label{main lifting theorem}
There exists a separable II$_1$ factor $M$ without property Gamma which satisfies the following. For every countably cofinal ultrafilter $\mathcal U$  on a set $I$ and $u_1,u_2\in\mathcal U(M^{\mathcal U})$ with $u_1^2=u_2^3=1$ and $\{u_1\}''\perp\{u_2\}''$, there exist Haar unitaries $v_1,v_2\in M^{\mathcal U}$ such that $[u_1,v_1]=[u_2,v_2]=[v_1,v_2]=0$.
\end{main}

Two von Neumann subalgebras $P,Q$ of a tracial von Neumann algebra $(M,\tau)$ are called {orthogonal} (written $P\perp Q$) if $\tau(xy)=\tau(x)\tau(y)$, for every $x\in P,y\in Q$. 
 For the notion of a countably cofinal ultrafilter, see Definition \ref{cofinal}. Here we only note that every free ultrafilter on  $\mathbb N$ is countably cofinal.

The  construction in Theorem \ref{main lifting theorem} is designed to imply the following estimate  for the 1-bounded entropy, which we present as our next main theorem.



\begin{main}
\label{h=0}
Let $M$ be  any II$_1$ factor satisfying the properties of Theorem \ref{main lifting theorem}. Then $h(M^{\mathcal U})\leq 0$, for every ultrafilter $\mathcal U$ on a set $I$.
\end{main}

The above estimate  allows us to prove   the desired non-isomorphism of ultrapowers. Indeed, let $M$ be as in Theorem \ref{main lifting theorem}.
If $(N,\tau)$ is a tracial von Neumann algebra which is elementarily equivalent to $M$, then $M^{\mathcal U}\cong N^{\mathcal V}$, for some ultrafilters $\mathcal U,\mathcal V$. Properties of the $1$-bounded entropy give that $h(N)\leq h(N^{\mathcal V})$ (see Facts \ref{fact 1} and \ref{in the presence of the ultra}). The conclusion of Theorem \ref{corollary about non ee of free group factors with our thing} then follows immediately. We refer the reader to Remark \ref{GJ}, pointed out to us by I. Goldbring and D. Jekel, for an explicit sequence which differentiates the elementary classes of $M$ and $N$.


Note that if $M$ is a II$_1$ factor with property Gamma, then $h(M^{\mathcal U})\leq 0$, for every ultrafilter $\mathcal U$ on a set $I$. Prior to the writing of this paper no examples of non-Gamma  II$_1$ factors  which satisfy this inequality were known.   Hence, Theorem \ref{h=0} is also of independent interest.

  



A II$_1$ factor is called {\it pseudocompact} if it is elementarily equivalent to a matrix ultraproduct (see \cite[Section 5]{FHS}). Pseudocompact factors cannot have property Gamma by \cite[Section 4]{FH} and  \cite[Theorem 5.1]{FHS}.
By combining Theorem \ref{h=0} with recent work of Jekel \cite{jekel} on matrix ultraproducts we obtain the first example of a non-Gamma II$_1$ factor which is not pseudocompact.

\begin{mcor}\label{Jekel corollary}
 There exists a separable II$_1$ factor $M$ without property Gamma which is not elementarily equivalent to $\prod_{\mathcal U}\mathbb M_{k_n}(\mathbb C)$, for any sequence $(k_n)\subset \mathbb N$ and any free ultrafilter $\mathcal U$ on $\mathbb N$.
\end{mcor}

\begin{remark}
The Connes Embedding Problem (CEP) asks if every separable II$_1$ factor embeds into $R^{\mathcal U}$, where $\mathcal U$ is a free ultrafilter on $\mathbb N$ \cite{Connes}. A negative answer to the CEP has been announced in the preprint \cite{CEP}.
Assuming $M_0$ is a  non-Connes-embeddable separable II$_1$ factor, then  $M=M_0*\text{L}(\mathbb Z)$ is a non-Gamma separable II$_1$ factor which is still not embeddable. Any such $M$ is  neither elementarily equivalent to any embeddable non-Gamma II$_1$ factor (e.g., $\text{L}(\mathbb F_2)$) nor pseudocompact. Moreover, assuming a negative answer to the CEP, \cite[Corollary 5.5]{GH2020} implies the existence of infinitely many elementary equivalence classes of non-Gamma II$_1$ factors. 
In contrast, our construction of a non-Gamma II$_1$ factor which is not elementarily equivalent to $\text{L}(\mathbb F_2)$ and not pseudocompact is explicit 
and does not depend on the answer to the CEP, nor does it use techniques from \cite{CEP}. We note that it is open whether the II$_1$ factor we construct is Connes-embeddable. As such, it remains an open question to find examples of Connes-embeddable non-Gamma II$_1$ factors which are not elementarily equivalent. 
\end{remark}

\subsection*{Comments on the proofs of Theorems \ref{main lifting theorem} and \ref{h=0}}

The proof of Theorem \ref{main lifting theorem} relies on a new construction of II$_1$ factors which is of independent interest and is presented in Section \ref{construction}. This associates, via a $2$-step amalgamated free product procedure, to every II$_1$ factor $M_1$ and unitaries $u_1,u_2\in M_1$,  a tracial von Neumann algebra  $M_2$ generated by $M_1$ and Haar unitaries $v_1,v_2\in M_2$ satisfying  $[u_1,v_1]=[u_2,v_2]=[v_1,v_2]=0$.
When $\{u_1\}''\perp\{u_2\}''$, we use deformation/rigidity results from \cite{IPP}  to deduce that $M_2$ is a II$_1$ factor. Moreover, under this assumption, we show that any irreducible subfactor $Q\subset M_1$ is still irreducible in $M_2$, see Theorem \ref{amalgam}.

In Section \ref{proofs of main results},  assuming that  $M_1$ has property (T) and iterating the above construction, we get an increasing sequence of II$_1$ factors $(M_n)_{n\geq 1}$ whose inductive limit  $M:=(\cup_{n \geq 1}M_n)''$
is non-Gamma and has the following property. For a countable dense set of unitaries $u_1,u_2\in M$
 with $u_1^2=u_2^3=1$ and $\{u_1\}''\perp\{u_2\}''$ there are Haar unitaries $v_1,v_2\in M$ such that $[u_1,v_1]=[u_2,v_2]=[v_1,v_2]=0$.
 Using a result which allows us to lift unitaries $u_1,u_2\in M^{\mathcal U}$ with $u_1^2=u_2^3=1$ and $\{u_1\}''\perp\{u_2\}''$ (see Lemma \ref{lift}) we conclude that $M^{\mathcal U}$ satisfies the conclusion of Theorem \ref{main lifting theorem}.  The restriction to unitaries $u_1$ and $u_2$ of orders $2$ and $3$ is due to the fact that Lemma \ref{lift} only applies in this case.

 The statement of Theorem \ref{main lifting theorem} is partially inspired by  \cite[Corollary 4.8]{Hayes2018}. This shows that if a diffuse tracial von Neumann algebra $(M,\tau)$ has property (C$'$) introduced in \cite[Definition 3.6]{GalatanPopa}, then $h(M)\leq 0$. 
 In particular, \cite[Corollary 4.8]{Hayes2018} implies that $h(M)\leq 0$, for any diffuse von Neumann algebra $(M,\tau)$ that is generated by $u_1,\cdots,u_k\in\mathcal U(M)$ so that there exist pairwise commuting Haar unitaries $v_1,\cdots,v_k\in\mathcal U(M^{\mathcal U})$ with $[u_i,v_i]=0$, for any $1\leq i\leq k$. 
Property (C$'$) is an asymptotic commutativity property which weakens Popa's property (C) \cite{Popa83}. The latter, itself a weakening of property Gamma, was shown to fail for $\text{L}(\mathbb F_n)$, $2\leq n\leq \infty$, in \cite{Dykema97}. 

To outline the proof of Theorem \ref{h=0}, let $M$ be as in Theorem \ref{main lifting theorem} and $\mathcal U$ be a countably cofinal ultrafilter on a set $I$. 
Using an observation made in the proof of \cite[Corollary 4.8]{Hayes2018} (see Fact \ref{twounitaries}) we derive that
$h(\{u_1,u_2\}'':M^{\mathcal U})\leq 0$, for any $u_1,u_2\in\mathcal U(M^{\mathcal U})$ with $u_1^2=u_2^3=1$ and $\{u_1\}''\perp\{u_2\}''$. Here, $h(N:M)$ denotes the $1$-bounded entropy of $N$ in the presence of $M$, see Subsection \ref{1bdd}. On the other hand, 
$M^{\mathcal U}$ can be generated by a family of subalgebras of the form $\{u_1,u_2\}''$, where $u_1,u_2\in\mathcal U(M^{\mathcal U})$ satisfy $u_1^2=u_2^3=1$ and $\{u_1\}''\perp\{u_2\}''$, all containing a fixed diffuse subalgebra. Using the behavior of the $1$-bounded entropy with respect to joins (see Facts \ref{unions} and \ref{joins}), we conclude that $h(M^{\mathcal U})\leq 0$, for any countably cofinal ultrafilter $\mathcal U$. 
Since $h(M)\leq 0$ and $M^{\mathcal U}\cong M$ for any ultrafilter $\mathcal U$ that is not countably cofinal, Theorem \ref{h=0} follows.

\subsection*{Acknowledgements:} We thank  Isaac Goldbring, David Jekel, Jesse Peterson, Sorin Popa and Stefaan Vaes for helpful comments.  

\section{Preliminaries}

\subsection{Tracial von Neumann algebras} 
Let $(M,\tau)$ be a tracial von Neumann algebra, i.e., a pair consisting of a von Neumann algebra $M$ and a faithful normal tracial state $\tau:M\rightarrow\mathbb C$. We denote by $\mathcal U(M)$ the group of unitaries of $M$ and by $M_{\text{sa}}$ the set of self-adjoint elements of $M$.
Given a self-adjoint set $S\subset M$, von Neumann's bicommutant theorem implies that $S''$ is the smallest unital von Neumann subalgebra of $M$ containing $S$. 
For von Neuman subalgebras $(M_\alpha)$ of $M$, we will use the notation  $\vee_\alpha M_\alpha$ for $(\cup_\alpha M_\alpha)''$.

For an ultrafilter $\mathcal U$ on a set $I$, we denote by $M^{\mathcal U}$ the tracial ultraproduct: the quotient $\ell^\infty(I,M)/\mathcal{J}$ by the closed ideal $\mathcal{J}\subset\ell^\infty(I,M)$  consisting of $x=(x_n)$ with $\lim\limits_{n\rightarrow\mathcal U}\|x_n\|_2= 0$.  
We have a natural diagonal inclusion $M\subset M^{\mathcal U}$ given by $x\mapsto (x_n)$, where $x_n=x$, for all $n\in I$.
A separable II$_1$ factor $M$ has {\it property Gamma} if  $M'\cap M^{\mathcal U}\neq \mathbb{C}1$, for a free  ultrafilter $\mathcal U$ on $\mathbb N$.  For more details on tracial ultraproducts, we refer the reader to  \cite[Appendix E]{BrownOzawa2008} and \cite[Section 5]{AP}. 

Two tracial von Neumann algebras $(M_1,\tau_1)$ and $(M_2,\tau_2)$ are said to be \emph{elementarily equivalent} if there exist ultrafilters $\mathcal U,\mathcal V$ on arbitrary sets $I,J$ such that $M_1^{\mathcal U} \cong M_2^{\mathcal V}$. This is the \emph{semantic} definition of elementary equivalence. The model theoretic (sometimes called syntactic) definition for elementary equivalence will not be stated in this paper, as it is equivalent to the semantic definition by deep results of Keisler-Shelah adapted to the continuous setting, see \cite[Section 2]{FHS} and \cite[Theorem 10.7]{hensoniovino}.


A key tool in our work is the amalgamated free product construction for tracial von Neumann algebras. 
Let $(M_1,\tau_1)$ and $(M_2,\tau_2)$ be tracial von Neumann algebras with a common von Neumann subalgebra $B$ such that ${\tau_1}_{|B}={\tau_2}_{|B}$. We denote by $M=M_1*_{B}M_2$ the amalgamated free product with its canonical trace $\tau$. See \cite{Popa93} and \cite{VDN1992} for more details on the construction.
 
 To prove that the II$_1$ factors we construct do not have property Gamma, we will use property (T) and Popa's intertwining techniques.
 
A II$_1$ factor has {\it property (T)}  \cite{ConnesJones} (see also \cite{Po01}) if for every $\varepsilon>0$, there are $F\subset M$ finite and $\delta>0$ such that for any Hilbert $M$-$M$-bimodule $\mathcal H$ and unit vector $\xi\in\mathcal H$ with $\max_{x\in F}\|x\xi-\xi x\|\leq\delta$, there exists $\eta\in\mathcal H$ satisfying $\|\eta-\xi\|\leq\varepsilon$ and $x\eta=\eta x$, for every $x\in M$.
Let $\Gamma$ be an icc countable group with property (T); for instance,  take $\Gamma=\text{SL}_3(\mathbb Z)$ by \cite{KazhdanTDef}. Then $M=\text{L}(\Gamma)$ is a II$_1
$ factor with property (T), see \cite[Theorem 2]{ConnesJones} and \cite[Theorem 4.1.7]{PopaCorr}.

In this paper, we will use the well-known fact that  II$_1$ factors with property (T) have weak spectral gap (in the sense of \cite{Popa_weakspgap}) in any inclusion:

\begin{proposition}\label{T fact}
Let $M$ be a II$_1$ factor and $M_1\subset M$ be a subfactor with property (T). Then $M_1'\cap M^{\mathcal U}= (M_1'\cap M)^{\mathcal U}$, for any  ultrafilter $\mathcal U$ on a set $I$. 
\end{proposition}

Conversely, if the equality $M_1'\cap M^{\mathcal U}=(M_1'\cap M)^{\mathcal U}$ holds for every II$_1$ factor $M$ containing $M_1$ and every ultrafilter $\mathcal U$ on $\mathbb  N$, then $M_1$ must have property (T), as shown recently in \cite{Tan}.


\begin{theorem}[see \cite{PopaStrongRigidity}]\label{popa intertwining} Let $(M,\tau)$ be a separable tracial von Neumann algebra and let $P\subset pM p,Q\subset M$ be von Neumann subalgebras. 
Then the following conditions  are equivalent:

\begin{enumerate}
\item There exist projections $p_0\in P, q_0\in Q$, a $*$-homomorphism $\theta:p_0P p_0\rightarrow q_0Q q_0$  and a non-zero partial isometry $v\in q_0M p_0$ such that $\theta(x)v=vx$, for all $x\in p_0P p_0$.


\item There is no sequence $u_n\in\mathcal U(P)$ satisfying $\|E_Q(x^*u_ny)\|_2\rightarrow 0$, for all $x,y\in pM$.
\end{enumerate}

\end{theorem}

If one of these equivalent conditions holds,  we write $P\prec_{M}Q$, and say that \emph{a corner of $P$ embeds into $Q$ inside $M$.}

\subsection{1-bounded entropy}\label{1bdd}


We recall some background for 1-bounded entropy theory (see \cite{Hayes2018}, \cite{Jung2007}) and direct the reader to \cite[Section 2.3]{FreePinsker} and \cite[Sections 2.2 and 2.3]{HJKE1} for a more detailed exposition. 
For a tracial von Neumann algebra $(M,\tau)$ and $X\in M_{\text{sa}}^d$, the {\it law} of $X$ is the linear functional $\ell_X:\mathbb{C}\langle t_1,\dots,t_d\rangle \to \mathbb{C}$ given by $\ell_X (f) = \tau(f(X))$.
Let $\Sigma_{d,R}$ be the set of all linear maps $\ell: \mathbb{C}\langle t_1,\dots,t_d\rangle \to \mathbb{C}$ satisfying that there exists a finite von Neumann algebra $(M,\tau)$ and $X \in M_{\text{sa}}^d$ such that $\ell=\ell_X$ and $\|x\| \leq R$ for all $x\in X$. We equip $\Sigma_{d,R}$ with the weak$^*$ topology.

We describe the orbital version of 1-bounded entropy (see Definition A.2 in \cite{Hayes2018}). Let $(M,\tau)$ be a diffuse tracial von Neumann algebra, and $X, Y\subset M_{\text{sa}}$ finite such that $\|x\| \leq R$ for all $x\in X\cup Y$.  Following \cite{VoiculescuFreeEntropy2}, for each weak$^*$ neighborhood $\mathcal{O}$ of $\ell_{X\sqcup Y}$ in $\Sigma_{d,R}$ and $n \in \mathbb{N}$, we define
\[
\Gamma_R^{(n)}(X:Y; \mathcal{O}) =  \{A \in \mathbb{M}_n(\mathbb{C})_{\text{sa}}^{X}: \exists B\in \mathbb{M}_n(\mathbb{C})_{\text{sa}}^{Y} \text{ such that }  \ell_{A\sqcup B} \in \mathcal{O}, \|A_x\|,\|B_y\|\leq R, \forall x\in X,y\in Y \}.
\]

Given $d,n\in \mathbb{N}$, $\varepsilon>0$ and $\Omega,\Xi\subseteq \mathbb{M}_{n}(\mathbb{C})^{d}$, then $\Xi$ is said to \emph{$(\varepsilon,\|\cdot\|_2)$-cover $\Omega$} if for every $A\in \Omega$, there is $B\in \Xi$ with $\|A-B\|_{2}<\varepsilon$. Define the \emph{covering number} $K_{\varepsilon}(\Omega,\|\cdot\|_{2})$ of $\Omega\subseteq \mathbb{M}_{n}(\mathbb{C})^{d}$ as the minimal cardinality of a set that $(\varepsilon,\|\cdot\|_{2})$-covers $\Omega$. 
We say that $\Xi$ \emph{orbitally $(\varepsilon,\|\cdot\|_2)$-covers} $\Omega$ if for every $A\in\Omega$, there is a $B \in \Xi$ and an $n \times n$ unitary matrix $V$ so that
$\|A-VBV^*\|_2 < \varepsilon$.
Define the \emph{orbital covering number} $K^{\textnormal{orb}}_\varepsilon(\Omega,\|\cdot\|_2)$ as the minimal cardinality of a set that orbitally $(\varepsilon,\|\cdot\|_2)$-covers $\Omega$.

Let $X_0,Y_0\subset M_{\text{sa}}$ not necessarily finite, satisfying $X_0''\subset Y_0''$ and $\|x\|\leq R$ for all $x\in X_0\cup Y_0$. Let $X,Y$ be finite subsets of $X_0$, $Y_0$ respectively. For a weak$^{*}$-neighborhood $\mathcal{O}$ of $\ell_{X\sqcup Y}$, we define
\begin{align*}
h_{\varepsilon}(X:Y;\mathcal{O})& := \limsup_{n \to\infty}\frac{1}{n^2}\log K_{\varepsilon}^{\textnormal{orb}}(\Gamma_{R}^{(n)}(X:Y;\mathcal{O})), \\
h_{\varepsilon}(X:Y) &:= \inf_{\mathcal{O} \ni \ell_{X\sqcup Y}} h_{\varepsilon}(\mathcal{O}),\\ h_{\varepsilon}(X_0:Y_0)& := \sup_{X\subset_{\text{finite}} X_0} \inf_{Y\subset_{\text{finite}} Y_0} h_{\varepsilon}(X:Y)\\ h(X_0:Y_0)& := \sup_{\epsilon>0} h_{\varepsilon}(X_0:Y_0)
\end{align*}

Note that $h(X_1:Y_1)=h(X_2:Y_2)$ if $X_1''=X_2''$ and $Y_1''=Y_2''$ by \cite[Theorem A.9]{Hayes2018}. Hence, given a von Neumann subalgebra $N\subset M$, we unambiguously write $h(N:M)$ (and call it the {\it $1$-bounded entropy of $N$ in the presence of $M$}) to be $h(X:Y)$ for some generating sets $X$ of $N $ and $Y$ of $M$. 
We write $h(M)=h(M:M)$ and call it the {\it $1$-bounded entropy of $M$}.
 
 For the purposes of this article we recall the following facts about $h$: 
 
 \begin{fact}(see  \cite[2.3.3]{HJKE1})\label{fact 1}
 $h(N_{1}:M_{1})\leq h(N_{2}:M_{2})$ if $N_{1}\subset N_{2}\subset M_{2}\subset M_{1}$ and $N_{1}$ is diffuse.
 \end{fact}

  \begin{fact}(see \cite[Proposition 4.5]{Hayes2018})\label{in the presence of the ultra}
 $h(N:M)=h(N:M^{\mathcal U})$ if $N\subset M$ is diffuse, and $\mathcal U$ is an ultrafilter on a set $I$. (Note that \cite[Proposition 4.5]{Hayes2018} asserts this fact for free ultrafilters $\mathcal U$. The fact is trivially true also for non-free (i.e., principal) ultrafilters.) \end{fact}

 \begin{fact}\label{joins}(see \cite[Lemma A.12]{Hayes2018})
  $h(N_1\vee N_2:M)\leq h(N_1:M)+h(N_2:M)$ if $N_1,N_2\subset M$ and $N_1\cap N_2$ is diffuse.  
  In particular, $h(N_1\vee N_2)\leq h(N_1)+h(N_2)$.
 \end{fact}

 \begin{fact}\label{unions}(see \cite[Lemma A.10]{Hayes2018})
  Assume that $(N_\alpha)_\alpha$ is an increasing chain of diffuse von Neumann subalgebras of $M$. Then $h(\bigvee_\alpha N_\alpha:M)=\sup_\alpha h(N_\alpha:M)$.
 \end{fact}

By \cite[Corollary 3.5]{Jung2007} and \cite[Proosition A.16]{Hayes2018}, $h(N)=\infty$ whenever $(N,\tau)$ is a tracial von Neumann algebra admitting a finite generating set $X\subset N_{\text{sa}}$ with  $\delta_0(X)>1$, where $\delta_0 $ is Voiculescu's modified free entropy dimension (see Section 6 of \cite{Voiculescu1996}). 


\begin{fact}\label{examples of h>0}
The following tracial von Neumann algebras $(N,\tau)$ satisfy $h(N)>0$. The first five examples all arise from identifying generating sets $X$ satisfying $\delta_0(X)>1$, and thus $h(N)=\infty$.  

\begin{enumerate}
    \item (see \cite[Lemma 3.7]{Jung2007})) $N_1*N_2$ where $(N_1,\tau_1)$ and $(N_2,\tau_2)$ are Connes-embeddable diffuse tracial von Neumann algebras.
    \item The free perturbation algebras of Voiculescu (see Theorem 4.1 in \cite{Brown}). 
    \item Many examples of amalgamated free products  $N_1*_{B}N_2$ where $B$ is  amenable (see Section 4 of \cite{BDJ2008} for precise examples).
    \item (see \cite{Aimsquare}) Graph products of finite dimensional tracial von Neumann algebras over trees where the cardinality of the vertex set is greater than or equal to 4.
  
    \item (see \cite{Shlyakhtenkononinnercocycles}, Theorem 3) Von Neumann algebras of Connes-embeddable nonamenable groups $\Gamma$ admitting non inner cocycles $c:\Gamma\rightarrow\mathbb C\Gamma$. 
      \item (see \cite{BenPT}, \cite{PTkilled}, \cite{PTcor}) Nonamenable von Neumann subalgebras of $\text{L}(\mathbb{F}_t)$ for $t>0$.
\end{enumerate}
\end{fact}

 The following recent result of Jekel provides another family of examples:

  \begin{fact}(see \cite[Theorem 1.1]{jekel})\label{Jekel fact} Suppose that $h(N)>0$.
Let  $\{n_k\}_{k=1}^{\infty}$ be an increasing sequence of natural numbers and $\mathcal U$ be a free ultrafilter on $\mathbb N$. Let $\mathcal{M}=\prod_{\mathcal U} \mathbb{M}_{n_k}(\mathbb{C})$. 
Then there exists an embedding $N\hookrightarrow \mathcal{M}$ such that $h(N: \mathcal{M})>0$. In particular $h(\mathcal{M})>0$.  \end{fact}






The following fact follows easily from Fact \ref{joins}. This  observation appears in the proof of Corollary 4.8 in \cite{Hayes2018}. For completeness, we include a proof here.
\begin{fact}\label{twounitaries}
Assume that  $u_1,u_2\in \mathcal U(M)$ such that there are Haar unitaries $v_1,v_2\in M$ satisfying $[v_1,u_1]= [v_2,u_2]=[v_1,v_2]=0$. Then $h(\{u_1,u_2\}'':M)\leq 0$.
\end{fact}

\begin{proof} Since $\{u_1,v_1\}'',\{v_1,v_2\}'',\{v_2,u_2\}''$ are abelian, we get   $$h(\{u_1,v_1\}'')=h(\{v_1,v_2\}'')=h(\{v_2,u_2\}'')=0.$$ Since $\{v_1\}''$ and $\{v_2\}''$ are diffuse, Fact \ref{joins} implies that $$h(\{u_1,u_2,v_1,v_2\}'')=h(\{u_1,v_1\}''\bigvee \{v_1,v_2\}''\bigvee \{v_2,u_2\}'')\leq 0.$$ Hence, using Fact \ref{fact 1} we see that $$h(\{u_1,u_2\}'':M)\leq h(\{u_1,u_2\}'': \{u_1,u_2,v_1,v_2\}'')\leq h(\{u_1,u_2,v_1,v_2\}'')\leq 0,$$
which proves the fact.
\end{proof}

\section{A lifting lemma}

The goal of this section is to establish the following lifting lemma:

\begin{lemma}\label{lift} Let $I$ be a set, $\mathcal U$ an ultrafilter on $I$ and $(M_n)_{n\in I}$ be a family of  II$_1$ factors. Consider projections 
 $p,q_1,q_2,q_3\in\prod_{\mathcal U} M_n$ such that $q_1+q_2+q_3=1$ and $\{p\}''\perp\{q_1,q_2,q_3\}''$. 

Then we can represent $p=(p_n)$ and $q_i=(q_{i,n})$, where $p_n,q_{i,n}\in M_n$ are projections such that $q_{1,n}+q_{2,n}+q_{3,n}=1$ and $\{p_n\}''\perp\{q_{1,n},q_{2,n},q_{3,n}\}''$, for every $n\in I$.
\end{lemma}

Lemma \ref{lift} is an immediate consequence of the following perturbation lemma. 

\begin{lemma}\label{lift2}
 For every $\varepsilon>0$, there exists $\delta=\delta(\varepsilon)>0$ such that the following holds. 
 
 Let $M$ be a II$_1$ factor  and $e,f_1,f_2,f_3\in M$ be projections such that $f_1+f_2+f_3=1$ and we have $|\tau(ef_i)-\tau(e)\tau(f_i)|\leq\delta$, for every $1\leq i\leq 3$. 
 Then there exist projections $p,q_1,q_2,q_3\in M$ such that $q_1+q_2+q_3=1$, $\|p-e\|_1\leq\varepsilon$, $\|q_i-f_i\|_1\leq\varepsilon$ and $\tau(pq_i)=\tau(p)\tau(q_i)$, for every $1\leq i\leq 3$.
\end{lemma}

Note that if $p,q$ are projections in a II$_1$ factor $M$, then $\|p-q\|_2=\|p(p-q)+(p-q)q\|_2\leq 2\|p-q\|_1$. This implies that the statement of Lemma \ref{lift2} still holds  if we replace $\|\cdot\|_1$ by $\|\cdot\|_2$. Using this observation, it is standard to derive Lemma \ref{lift} from Lemma \ref{lift2}.



The proof of Lemma \ref{lift2} is based on the next two lemmas.

\begin{lemma}\label{excise}
Let $(M,\tau)$ be a diffuse tracial von Neumann algebra, $\delta>\varepsilon>0$ and $x=x^*\in M$ with $|\tau(x)|\leq\varepsilon$ and $\|x\|_1>\delta$.
Then there is a projection $p\in M$ such that $\tau(xp)=0$ and $\tau(p)>\frac{\delta-\varepsilon}{\delta+\varepsilon}$.
\end{lemma}

\begin{proof}
Let $x=y-z$ be the decomposition of $x$ into its positive and negative parts and $q\in M$ be the support projection of $y$. Then $y\in qMq$ and $z\in (1-q)M(1-q)$.
If $\tau(x)=0$, there is nothing to prove.
We may assume that $\tau(x)>0$, since the case $\tau(x)<0$ is  analogous.

Since $\tau(y)-\tau(z)=\tau(x)\leq\varepsilon$ and $\tau(y)+\tau(z)=\|x\|_1>\delta$, letting $s=\frac{\delta-\varepsilon}{\delta+\varepsilon}\in (0,1)$, it follows that $\tau(y)s<\tau(z)$.
Let $y'\in qMq$ be a self-adjoint operator with finite spectrum such that \begin{equation}\label{y'}2\|y'-y\|_1<\tau(z)-\tau(y)s.\end{equation}

Since $M$ is diffuse and $y'$ has finite spectrum, we can find an increasing net of projections  $(e_t)_{t\in [0,1]}$ in $qMq$  such that $e_0=0,e_1=q,\tau(e_t)=\tau(q)t$ and $\tau(y'e_t)=\tau(y')t$, for every $t\in [0,1]$. Then for every $t\in [0,1]$, we have that
$$|\tau(ye_t)-\tau(y)t|\leq |\tau(ye_t)-\tau(y'e_t)|+|(\tau(y)-\tau(y'))t|\leq 2\|y'-y\|_1,$$ and thus $\tau(ye_t)\leq\tau(y)t+2\|y'-y\|_1$. 

Combining this inequality for $t=s$ with \eqref{y'} gives that 
$\tau(ye_s)<\tau(z)$. As $\tau(ye_1)=\tau(y)>\tau(z)$ and the map $t\mapsto\tau(ye_t)$ is continuous, we can find $t\in (s,1)$ such that $\tau(ye_t)=\tau(z)$. 

Finally, let $p=e_t+(1-q)$. Then we have $\tau(xp)=\tau(yp)-\tau(zp)=\tau(ye_t)-\tau(z)=0$ and $\tau(p)=\tau(e_t)+\tau(1-q)=t\tau(q)+\tau(1-q)\geq t>s$, which finishes the proof.
\end{proof}

\begin{lemma}\label{proj}
Let $\varepsilon,\delta\geq 0$ such that $\varepsilon<\delta^2$ and  $(M,\tau)$ be a diffuse tracial von Neumann algebra. Let $p,f_1,f_2,f_3\in M$ be projections such that   $f_1+f_2+f_3=1$,  $|\tau(pf_i)-\tau(p)\tau(f_i)|\leq \varepsilon$ and $\|f_i(p-\tau(p))f_i\|_1>\delta$, for every $1\leq i\leq 2$.

Then there exist projections $q_1,q_2,q_3\in M$ such that $q_1+q_2+q_3=1, \tau(pq_i)=\tau(p)\tau(q_i)$ and $\|q_i-f_i\|_1<\frac{4\varepsilon}{\delta^2}$, for every $1\leq i\leq 3$.
\end{lemma}

\begin{proof} 
Let $1\leq i\leq 2$ and define $x_i=f_i(p-\tau(p))f_i$. Then we have $x_i=x_i^*\in f_iMf_i$ and $|\tau(x_i)|=|\tau(pf_i)-\tau(p)\tau(f_i)|\leq\varepsilon$. Since $\|x_i\|_1>\delta$ and $\|x_i\|_1\leq\tau(f_i)\|p-\tau(p)\|\leq\tau(f_i)$, we get that $\tau(f_i)>\delta$. Thus, $|\frac{\tau(x_i)}{\tau(f_i)}|\leq\frac{\varepsilon}{\delta}$ and $\frac{\|x_i\|_1}{\tau(f_i)}\geq \|x_i\|_1>\delta$.
Altogether, by applying Lemma \ref{excise} to $x_i\in f_iMf_i$, we find a projection $q_i\in f_iMf_i$ such that 
\begin{equation}\label{r_i}\text{$\tau(x_iq_i)=0$ \;\;\;\;\; and \;\;\;\;\; $\frac{\tau(q_i)}{\tau(f_i)}>\frac{\delta-\frac{\varepsilon}{\delta}}{\delta+\frac{\varepsilon}{\delta}}=\frac{1-\frac{\varepsilon}{\delta^2}}{1+\frac{\varepsilon}{\delta^2}}>1-\frac{2\varepsilon}{\delta^2}.$}\end{equation}
Using \eqref{r_i} we get that $\tau((p-\tau(p))q_i)=\tau(x_iq_i)=0$ and thus $\tau(pq_i)=\tau(p)\tau(q_i)$. Moreover, $$\|q_i-f_i\|_1=\tau(f_i)-\tau(q_i)<\frac{2\varepsilon}{\delta^2}\tau(f_i)\leq\frac{2\varepsilon}{\delta^2}.$$
Let $q_3=1-q_1-q_2$ \textcolor{red}{and $f_3= 1- f_1-f_2$}. Then  $\tau(pq_3)=\tau(p)-\tau(pq_1)-\tau(pq_2)=\tau(p)(1-\tau(q_1)-\tau(q_2))=\tau(p)\tau(q_3)$. Moreover, $\|q_3-f_3\|_1=\|(q_1+q_2)-(f_1+f_2)\|_1\leq \|q_1-f_1\|_1+\|q_2-f_2\|_1<\frac{4\varepsilon}{\delta^2}$. This finishes the proof of the lemma.
\end{proof}

\begin{proof}[Proof of Lemma \ref{lift2}] Assume that the conclusion of Lemma \ref{lift2} fails.
Then there is $\varepsilon>0$ such that for every $n\in\mathbb N$ we can find a II$_1$ factor $(M_n,\tau_n)$ and projections $e_n,f_{1,n},f_{2,n},f_{3,n}\in M_n$ satisfying the following: $f_{1,n}+f_{2,n}+f_{3,n}=1$, $|\tau_n(e_nf_{i,n})-\tau_n(e_n)\tau_n(f_{i,n})|\leq\frac{1}{n}$, for every $1\leq i\leq 3$, and $\|p_n-e_n\|_1+\|q_{1,n}-f_{1,n}\|_1+\|q_{2,n}-f_{2,n}\|_1+\|q_{3,n}-f_{3,n}\|_1>\varepsilon$, for all projections $p_n,q_{1,n},q_{2,n},q_{3,n}\in M_n$ such that $q_{1,n}+q_{2,n}+q_{3,n}=1$ and $\tau_n(p_nq_{i,n})=\tau_n(p_n)\tau_n(q_{i,n})$, for every $1\leq i\leq 3$. 

Let $\mathcal U$ be a free ultrafilter on $\mathbb N$.
 Let $\tau$ be the canonical  trace of $\prod_\mathcal UM_n$ given by $\tau(x)=\lim\limits_{n\rightarrow\mathcal U}\tau_n(x_n)$, for every $x=(x_n)\in\prod_{\mathcal U}M_n$. Then $p=(e_n), q_1=(f_{1,n}),q_2=(f_{2,n}),q_3=(f_{3,n})\in\prod_\mathcal U M_n$ are projections satisfying that $q_1+q_2+q_3=1$ and $\{p\}''\perp\{q_1,q_2,q_3\}''$.

We will get a contradiction by analyzing two cases:

{\bf Case 1}.  The set $\{1\leq i\leq 3\mid q_i(p-\tau(p))q_i=0\}$ has at most one element. 

Without loss of generality,  assume that $q_i(p-\tau(p))q_i\not=0$, for all $1\leq i\leq 2$.

 For $n\in\mathbb N$ and $1\leq i\leq 2$, define $\delta_i=\|q_i(p-\tau(p))q_i\|_1,\delta_{i,n}=\|f_{i,n}(e_n-\tau_n(e_n))f_{i,n}\|_1$ and $\kappa_{i,n}=|\tau_n(e_nf_{i,n})-\tau_n(e_n)\tau_n(f_{i,n})|$. Then $\delta_i>0$, $\lim\limits_{n\rightarrow\mathcal U}\delta_{i,n}=\delta_i$ and $0\leq\kappa_{i,n}\leq\frac{1}{n}$, for every $n\in\mathbb N$. Let $\delta=\min\{\delta_1,\delta_2\}$. Then the set $J$ of $n\in \mathbb N$ such that $\delta_{i,n}>\frac{\delta}{2}$ and $\kappa_{i,n}<\delta_{i,n}^2$, for every $1\leq i\leq 2$, belongs to $\mathcal U$.

By Lemma \ref{proj}, for every $n\in J$, we find projections $q_{i,n}\in M_n$ such that $q_{1,n}+q_{2,n}+q_{3,n}=1$, $\tau_n(e_nq_{i,n})=\tau_n(e_n)\tau_n(q_{i,n})$ and $\|q_{i,n}-f_{i,n}\|_1<\frac{4\kappa_{i,n}}{\delta_{i,n}^2}<
\frac{16}{\delta^2n^2}$, for every $1\leq i\leq 2$. As $J$ is infinite, we can find $n\in J$ such that $\frac{16}{\delta^2n^2}<\frac{\varepsilon}{3}$, for every $1\leq i\leq 2$. Put $p_n=e_n$. Then $\|p_n-e_n\|_1+\|q_{1,n}-f_{1,n}\|_1+\|q_{2,n}-f_{2,n}\|_1+\|q_{3,n}-f_{3,n}\|_1<\varepsilon$,  contradicting the first paragraph of the proof.

{\bf Case 2}.  The set $\{1\leq i\leq 3\mid q_i(p-\tau(p))q_i=0\}$ has at least two elements. 

Without loss of generality,  assume that $q_i(p-\tau(p))q_i=0$, for every $1\leq i\leq 2$.

We claim that $Q:=\{p,q_1,q_2,q_3\}''$ is a type I von Neumann algebra. Let $1\leq i\leq 2$. Since $q_ipq_i=\tau(p)q_i$, we get that $v_i:=\tau(p)^{-\frac{1}{2}}q_ip$ is a partial isometry. Thus, $p_i:=v_i^*v_i=\tau(p)^{-1}pq_ip$ is a projection. Recall that any von Neumann algebra generated by two projections is of type I, being a direct sum of type I$_1$ and I$_2$ algebras.
Since  $pQp=\{pq_1p,pq_2p,pq_3p\}''=\{p_1,p_2,p\}''$ and $p_1,p_2\in p(\prod_\mathcal UM_n) p$ are projections, $pQp$ is of type I. Since $q_i((1-p)-\tau(1-p))q_i=q_i(\tau(p)-p)q_i=0$, for every $1\leq i\leq 2$, we also get that $(1-p)Q(1-p)$ is of type I. The last two facts imply the claim. 

 Next, endow $Q\subset\prod_{\mathcal U}M_n$ with the restriction of $\tau$ to $Q$. 
Since $Q$ is of type I, it is hyperfinite. 
If $n\in\mathbb N$, then using that $M_n$ is a II$_1$ factor we can find a normal $*$-homomorphism $\pi_n:Q\rightarrow M_n$ such that $\tau_n(\pi_n(x))=\tau(x)$, for every $x\in Q$. Then the normal $*$-homomorphism $\pi:Q\rightarrow\prod_\mathcal UM_n$ given by $\pi(x)=(\pi_n(x))$ satisfies that $\tau(\pi(x))=\lim\limits_{n\rightarrow\mathcal U}\tau_n(\pi_n(x))=\tau(x)$, for every $x\in Q$.
As is well-known (see, e.g, \cite[Theorem 1.1]{HT}), since $Q$ is hyperfinite, any two trace-preserving $*$-homomorphism from $Q$ to $\prod_\mathcal UM_n$ are unitarily conjugate. Thus,  we can find $u_n\in\mathcal U(M_n)$, for every $n\in\mathbb N$, such that $x=(u_n\pi_n(x)u_n^*)$, for every $x\in Q$.
In particular,  $p=(p_n)$ and $q_i=(q_{i,n})$, where $p_n=u_n\pi_n(p)u_n^*$ and $q_{i,n}=u_n\pi_n(q_i)u_n^*$, for every $n\in\mathbb N$
 and $1\leq i\leq 3$. 
 Then $q_{1,n}+q_{2,n}+q_{3,n}=1$, for every $n\in\mathbb N$, and $\lim\limits_{n\rightarrow\mathcal U}(\|p_n-e_n\|_1+\|q_{1,n}-f_{1,n}\|_1+\|q_{2,n}-f_{2,n}\|_1+\|q_{3,n}-f_{3,n}\|_1)=0$. Moreover,  $$\tau_n(p_nq_{i,n})=\tau_n(\pi_n(pq_i))=\tau(pq_i)=\tau(p)\tau(q_i)=\tau_n(\pi_n(p))\tau_n(\pi_n(q_i))=\tau_n(p_n)\tau_n(q_{i,n}),$$ for every $n\in\mathbb N$ and $1\leq i\leq 3$. Altogether, this also contradicts the first paragraph of the proof.
 \end{proof}

\section{A construction of II$_1$ factors}\label{construction}
In this section, we introduce a new construction of II$_1$ factors which we will use iteratively to build the II$_1$ factor in Theorem \ref{main lifting theorem}.


\begin{definition}
Let $(M,\tau)$ be a tracial von Neumann algebra and $A_1,A_2\subset M$ be von Neumann subalgebras. We define a tracial von Neumann algebra  $\Phi(M,A_1,A_2)$ as follows. Put $B_1=B_2=\text{L}(\mathbb Z)$ and define  $$\Phi(M,A_1):=M*_{A_1}(A_1\overline{\otimes}B_1)\;\;\;\text{and}$$  $$\Phi(M,A_1,A_2):=\Phi(M,A_1)*_{(A_2\bigvee B_1)}((A_2\vee B_1)\overline{\otimes}B_2).$$
 Given $u_1,u_2\in\mathcal U(M)$, we will use the notation
  $\Phi(M,u_1,u_2):=\Phi(M,\{u_1\}'',\{u_2\}'').$
\end{definition}
More generally, given von Neumann subalgebras $A_1,\cdots A_k\subset M$ one can define $\Phi(M,A_1,\cdots, A_k)$ inductively by letting $B_1=\cdots=B_k=\text{L}(\mathbb Z)$ and for every $1\leq i\leq k$
$$\Phi(M,A_1,\cdots, A_i):=\Phi(M,A_1,\cdots,A_{i-1})*_{(A_i\bigvee B_1\bigvee\cdots\bigvee B_{i-1})}((A_i\vee B_1\vee\cdots\vee B_{i-1})\overline{\otimes}B_i).$$

Here, we focus on the case $k=2$ which suffices for the purpose of proving Theorem \ref{main lifting theorem}.
The main result of this section gives necessary conditions which guarantee that $\Phi(M,A_1,A_2)$ is a II$_1$ factor. Furthermore, we prove:

\begin{theorem}\label{amalgam}
Let $(M,\tau)$ be a tracial von Neumann algebra and $A_1,A_2\subset M$ be von Neumann subalgebras such that $A_1\perp A_2$ and $M\nprec_MA_i$, for every $i=1,2$. Put $P=\Phi(M,A_1,A_2)$.

 Then $P$ is a II$_1$ factor
containing Haar unitaries $v_1,v_2\in P$ so that $v_1\in A_1'\cap P,v_2\in A_2'\cap P$ and $[v_1,v_2]=0$. Moreover, if $Q\subset M$ is a von Neumann subalgebra such that $Q\nprec_{M}A_i$, for every $1\leq i\leq 2$, then $Q'\cap P\subset M$. 

\end{theorem}
In the proof of Theorem \ref{main lifting theorem}, we will use the following immediate corollary of Theorem \ref{amalgam}
\begin{corollary}\label{amalgam2}
Let $(M,\tau)$ be a tracial von Neumann algebra having no type I direct summand. Let $u_1,u_2\in\mathcal U(M)$ such that $\{u_1\}''\perp\{u_2\}''$ and put $P=\Phi(M,u_1,u_2)$.

 Then $P$ is a II$_1$ factor
containing Haar unitaries $v_1,v_2\in P$ so that $[u_1,v_1]=[u_2,v_2]=[v_1,v_2]=0$. Moreover, if $Q\subset M$ is a von Neumann subalgebra such that $Q\nprec_{M}\{u_i\}''$, for every $1\leq i\leq 2$, then $Q'\cap P\subset M$. 

\end{corollary}
Since $M$ has no type I direct summand, $M\nprec_M\{u_i\}''$, for every $1\leq i\leq 2$, and thus Corollary \ref{amalgam2} follows from Theorem \ref{amalgam}.

\begin{remark}
Let us argue that the condition that $\{u_1\}''\perp \{u_2\}''$ in Corollary \ref{amalgam2} is necessary in order to get that $M$ is a II$_1$ factor. Thus, the condition that $A_1\perp A_2$ in Theorem \ref{amalgam} is also necessary.
In the context of Corollary \ref{amalgam2}, assume that $M$ is generated by $u_2$ and $\text{E}_{\{u_1\}''}(u_2)$. 
Denote $N:=\Phi(M,\{u_1\}'')=M*_{\{u_1\}''}(\{u_1\}''\overline{\otimes}\text{L}(\mathbb Z))$ and let $v_1\in\text{L}(\mathbb Z)$ be a generating Haar unitary. By \cite[Theorem 1.1]{IPP} we get that  $\text{L}(\mathbb Z)'\cap N=\{u_1\}''\overline{\otimes}\text{L}(\mathbb Z)$. This gives that $$\text{E}_{\text{L}(\mathbb Z)'\cap N}(u_2)=\text{E}_{\{u_1\}''\overline{\otimes}\text{L}(\mathbb Z)}(u_2)=\text{E}_{\{u_1\}''}(u_2).$$ On the other hand,  $\text{E}_{\text{L}(\mathbb Z)'\cap N}(u_2)$ is the $\|\cdot\|_2$-limit of the sequence $(\frac{1}{n}\sum_{k=1}^nv_1^ku_2{v_1^*}^k)_n$ and thus belongs to $\{u_2,v_1\}''$. The last two facts together imply that $\text{E}_{\{u_1\}''}(u_2)\in \{u_2,v_1\}''$. Since $M$ is generated by $u_2$ and $\text{E}_{\{u_1\}''}(u_2)$, we get that $M\subset \{u_2,v_1\}''$. Since $N$ is generated by $M$ and $v_1$,  we get that $\{u_2,v_1\}''=N$. Thus, $\Phi(M,u_1,u_2)=N\overline{\otimes}\text{L}(\mathbb Z)$ is not a factor, so the conclusion of Corollary \ref{amalgam2} does not hold.

Now,  the existence of $u_1,u_2\in\mathcal U(M)$ such that $\{u_2,\text{E}_{\{u_1\}''}(u_2)\}''=M$,  can be checked whenever $M$ is generated by two unitaries $u_1,\widetilde{u_2}$ such that $\{u_1\}''\perp\{\widetilde{u_2}\}''$ (e.g., if $M=\text{L}(\Gamma)$, for any $2$-generated group $\Gamma$). To see this, write $\widetilde{u_2}=\exp(ih)$, where $h\in\{\widetilde{u_2}\}''$ is a self-adjoint element, let $n\in\mathbb N$ such that $\tau(\exp(\frac{ih}{n}))\not=0$ and define $u_2=u_1\exp(\frac{ih}{n})$. Then $\text{E}_{\{u_1\}''}(u_2)=\tau(\exp(\frac{ih}{n}))u_1$ and thus $\{u_2,\text{E}_{\{u_1\}''}(u_2)\}''=\{u_1,\exp(\frac{ih}{n})\}''=\{u_1,\widetilde{u_2}\}''=M$.
\end{remark}

The proof of Theorem \ref{amalgam} relies on the main technical result of \cite{IPP}.
To recall the latter result, let $(M_1,\tau_1)$ and $(M_2,\tau_2)$ be tracial von Neumann algebras with a common von Neumann subalgebra $B$ such that ${\tau_1}_{|B}={\tau_2}_{|B}$.
 Let $M=M_1*_{B}M_2$ be the amalgamated free product with its canonical trace $\tau$. By \cite[Section 5.1]{PV09}, for $0<\rho<1$ we have a unital tracial completely positive map $\text{m}_\rho:M\rightarrow M$ such that $\text{m}_\rho(b)=b$, for every $b\in B$, and $\text{m}_\rho(x_1x_2\cdots x_n)=\rho^nx_1x_2\cdots x_n$, for every $x_i\in M_{i_j}\ominus B$, where $i_j\in\{1,2\}$, for every $1\leq j\leq n$, and $i_j\not=i_{j+1}$, for every $1\leq j\leq n-1$. Then \begin{equation}\label{def}\text{$\lim\limits_{\rho\rightarrow 1}\|\text{m}_\rho(x)-x\|_2=0$, for every $x\in M$.}\end{equation}
 
 The following is the main technical result of \cite{IPP}, formulated here as in \cite[Theorem 5.4]{PV09}, see also \cite[Section 5]{Ho07}.

\begin{theorem}\label{IPP05}
 Let $(M_1,\tau_1)$ and $(M_2,\tau_2)$ be tracial von Neumann algebras with a common von Neumann subalgebra $B$ such that ${\tau_1}_{|B}={\tau_2}_{|B}$.
 Let $M=M_1*_{B}M_2$ be the amalgamated free product with its canonical trace $\tau$. 
 Let $Q\subset pMp$ be a von Neumann subalgebra. Assume that there are $0<\rho<1$ and $c>0$ such that $\|\emph{m}_\rho(u)\|_2\geq c$, for every $u\in\mathcal U(Q)$.
 
 Then $Q\prec_MM_1$ or $Q\prec_MM_2$.
\end{theorem}
As $\tau(\text{m}_{\rho^2}(u)u^*)=\|\text{m}_\rho(u)\|_2^2\geq c^2$, for every $u\in\mathcal U(Q)$, \cite[Theorem 5.4]{PV09} implies Theorem \ref{IPP05}.

\begin{lemma}\label{free} Let $(M_1,\tau_1)$ and $(M_2,\tau_2)$ be tracial von Neumann algebras with a common von Neumann subalgebra $B$ such that ${\tau_1}_{|B}={\tau_2}_{|B}$.
 Let $M=M_1*_{B}M_2$ be the amalgamated free product with its canonical trace $\tau$. For $i\in\{1,2\}$, let $A_i\subset M_i$ be a von Neumann subalgebra with $A_i\perp B$. 
Let $Q\subset M_1$ be a von Neumann subalgebra such that $Q\prec_{M}A_1\vee A_2$ and $Q\nprec_{M_1}B$. 

Then $Q\prec_{M_1}A_1$.
\end{lemma}

\begin{proof} Denote $A=A_1\vee A_2$. We first claim that $A_1$ and $A_2$ are freely independent inside $M$ and thus $A=A_1*A_2$.
Let  $a_j\in A_{i_j}\ominus\mathbb C1$ for $i_j\in\{1,2\}$, for every $1\leq j\leq n$, where $i_j\not=i_{j+1}$, for every $1\leq j\leq n-1$.
 Since $A_i\perp B$, for every $i\in\{1,2\}$, we get that $E_B(a_j)=0$, for every $1\leq j\leq n$. This implies that $\tau(a_1a_2\cdots a_n)=0$, proving the claim.
 
 Since $Q\prec_MA$, we can find projections $q\in Q,p\in A$, a nonzero partial isometry $v\in pMq$ and $*$-homomorphism $\varphi:qQq\rightarrow pAp$ such that $\varphi(x)v=vx$, for every $x\in qQq$.
 Moreover, we may assume that the support projection of
 $E_A(vv^*)$ is equal to $p$. 
 
 \begin{claim}\label{QA_1}
 $\varphi(qQq)\prec_A A_1$ or $\varphi(qQq)\prec_A A_2$.
 \end{claim}  
{\it Proof of Claim \ref{QA_1}.}
 Since $\text{m}_\rho$ is a unital tracial completely positive map, using \eqref{def} and \cite[Corollary, Section 1.1.2]{Po01} we deduce that \begin{equation}\label{bimod}
 \text{ $\sup_{x\in (M)_1}\|\text{m}_\rho(xv)-\text{m}_{\rho}(x)v\|_2
     \rightarrow 0$\;\; and\;\;  $\sup_{x\in (M)_1}\|\text{m}_\rho(vx)-v\text{m}_\rho(x)\|_2\rightarrow 0$,\;\; as $\rho\rightarrow 1$.}
 \end{equation}
 Now, if $x\in M_1$, then the definition of $\text{m}_\rho$ implies that $\text{m}_\rho(x)=\text{E}_B(x)+\rho(x-\text{E}_B(x))$ and thus $\|\text{m}_\rho(x)-x\|_2=(1-\rho)\|x-\text{E}_B(x)\|_2\leq (1-\rho)\|x\|_2$. In particular, since $Q\subset M_1$, we derive that 
 \begin{equation}
\label{Q}\text{$\sup_{x\in (qQq)_1}\|\text{m}_\rho(x)-x\|_2\rightarrow 0$,\;\; as $\rho\rightarrow 1$.}     
 \end{equation}
 By combining \eqref{bimod} and \eqref{Q} and using that  $\varphi(x)v=vx$, for every $x\in qQq$, it follows that $\sup_{x\in (qQq)_1}\|\text{m}_\rho(\varphi(x))v-vx\|_2\rightarrow0$, as $\rho\rightarrow 1$. Therefore, we can find $0<\rho<1$ such that $\|\text{m}_\rho(\varphi(x))v-vx\|_2<\|v\|_2/2$, for every $x\in (qQq)_1$. This implies that $$\text{$\|\text{m}_\rho(\varphi(u))\|_2\geq \|\text{m}_\rho(\varphi(u))v\|_2>\|v\|_2/2$,\;\; for every $u\in \mathcal U(qQq)$.}$$
In other words, $\|\text{m}_\rho(y)\|>\|v\|_2$, for every $y\in\mathcal U(\varphi(qQq))$. Note that the restriction of $\text{m}_\rho$ to $A$ is equal to the map $\text{m}_\rho$ on $A$ associated with the free product decomposition $A=A_1*A_2$. Since $\varphi(qQq)\subset pAp$, we can apply Theorem \ref{IPP05} to get the claim. \hfill$\square$

By Claim \ref{QA_1}, we have that $\varphi(qQq)\prec_AA_i$, for some $i\in\{1,2\}$.
Since the support projection of
 $E_A(vv^*)$ is equal to $p$,  \cite[Remark 3.8]{Va07}
implies that $Q\prec_MA_i$. Finally, since $qQq\subset M_1$, $A_i\subset M_i$ and $Q\nprec_{M_1}B$, applying \cite[Theorem 1.1]{IPP} gives that $i=1$ and $Q\prec_{M_1}A_1$.
 \end{proof}

\begin{proof}[Proof of Theorem \ref{amalgam}]
Let $P=\Phi(M,A_1,A_2)=N*_{(A_2\vee B_1)}((A_2\vee B_1)\overline{\otimes}B_2)$, where  $B_1=B_2=\text{L}(\mathbb Z)$ and $N=\Phi(M,A_1)=M*_{A_1}(A_1\overline{\otimes}B_1)$.
Let $v_1\in B_1$ and $v_2\in B_2$ be generating Haar unitaries. Since $[A_1,B_1]=[A_2,B_2]=[B_1,B_2]=\{0\}$, we get that $v_1\in A_1'\cap P,v_2\in A_2'\cap P$ and  $[v_1,v_2]=0$.

Next, we prove the moreover assertion. Let $Q\subset M$ be a von Neumann subalgebra such that $Q\nprec_{M}A_i$, for every $1\leq i\leq 2$. Since $N=M*_{A_1}(A_1\overline{\otimes}B_1)$, $A_2\perp A_1$, $B_1\perp A_1$ 
and $Q\nprec_M A_1$, by Lemma \ref{free} we conclude that
\begin{equation}\label{Q}
Q\nprec_N A_2\vee B_1.
\end{equation} 
Since $P=N*_{(A_2\vee B_1)}((A_2\vee B_1)\overline{\otimes}B_2)$, using \eqref{Q} and applying \cite[Theorem 1.1]{IPP} we get that $Q'\cap P\subset N$. Since $N=M*_{A_1}(A_1\overline{\otimes} B_1)$ and $Q\nprec_M A_1$, applying \cite[Theorem 1.1]{IPP} again gives that $Q'\cap N\subset M$. Altogether, we get that $Q'\cap P\subset M$, which proves the moreover assertion.

Since $M\nprec_M A_i$, for every $1\leq i\leq 2$. By applying the moreover assertion to $Q=M$, we get that
$M'\cap P\subset M$, hence $\mathcal Z(P)=P'\cap M\subset\mathcal Z(M)$. Thus, if $M$ is a II$_1$ factor, then $P$ is a II$_1$ factor. 

In the general case, we first note that
 \cite[Theorem 1.1]{IPP} gives that $B_1'\cap M=A_1$ and $B_2'\cap M=(B_2'\cap N)\cap M=(A_2\vee B_1)\cap M$.
Thus, $\mathcal Z(P)=P'\cap M\subset A_1\cap (A_2\vee B_1)$.
We claim that $A_1\perp (A_2\vee B_1)$. Assuming the claim, it follows that $A_1\cap (A_2\vee B_1)=\mathbb C1$ and so $P$ is a II$_1$ factor.
To justify the claim and finish the proof, denote $M_1=M, M_2=A_1\overline{\otimes} B_1$, $C_1=A_2$, $C_2=B_1$ and $B=A_1$. Thus, $N=M_1*_BM_2$ and the claim is equivalent to $B\perp (C_1\vee C_2)$. Let $x\in B$ and $y\in C_1\vee C_2$ of the form $y=y_1y_2\cdots y_n$, where $y_j\in C_{i_j}\ominus\mathbb C1$ for some $i_j\in\{1,2\}$, for every $1\leq j\leq n$, such that $i_j\not=i_{j+1}$, for every $1\leq j\leq n-1$. Since $C_i\perp B$, for every $1\leq i\leq 2$, we get that $\text{E}_B(y_j)=0$ and thus $y_j\in M_{i_j}\ominus B$, for every $1\leq j\leq n$. Moreover, $\text{E}_B(xy_1)=x\text{E}_B(y_1)=0$ and thus $xy_1\in M_{i_1}\ominus B$. This implies that $\tau(xy)=\tau((xy_1)y_2\cdots y_n)=0$. Since $C_1$ and $C_2$ are freely independent, as shown in the proof of Lemma \ref{free}, the linear span of elements $y\in C_1\vee C_2$ of the above form is dense in $(C_1\vee C_2)\ominus\mathbb C1$. Thus, we get that $B\perp (C_1\vee C_2)$, proving the claim.
\end{proof}

\section{Proofs of main results}\label{proofs of main results}
This section is devoted to the proofs of our main results. 

\subsection{Proof of Theorem \ref{main lifting theorem}}
We start by constructing the II$_1$ factor from Theorem \ref{main lifting theorem}
 by iterating the construction from Section  \ref{construction}.

For a II$_1$ factor $M$, we denote by $\mathcal V(M)$ the set of pairs $(u_1,u_2)\in\mathcal U(M)\times\mathcal U(M)$ such that $u_1^2=u_2^3=1$ and $\{u_1\}''\perp\{u_2\}''$. We endow $\mathcal U(M)\times\mathcal U(M)$ with the product $\|\cdot\|_2$-topology.

\begin{definition}\label{Mtilde}Let $M_1$ be a II$_1$ factor. 
We construct a new II$_1$ factor $M$ which contains $M_1$ and arises as the inductive limit of 
 a sequence $(M_n)_{n\in\mathbb N}$ of II$_1$ factors satisfying $M_n\subset M_{n+1}$, for every $n\in\mathbb N$. Let $\sigma=(\sigma_1,\sigma_2):\mathbb N\rightarrow\mathbb N\times\mathbb N$ be a bijection such that $\sigma_1(n)\leq n$, for every $n\in\mathbb N$. 
 Assume that $M_1,\ldots,M_n$ have been constructed, for some $n\in\mathbb N$.
Let $\{(u_1^{n,k},u_2^{n,k})\}_{k\in\mathbb N}\subset\mathcal V(M_n)$ be a $\|\cdot\|_2$-dense sequence.
We define $$M_{n+1}:=\Phi(M_n,u_1^{\sigma(n)},u_2^{\sigma(n)}).$$
Note that $M_{n+1}$ is well-defined since $\sigma_1(n)\leq n$ and thus $(u_1^{\sigma(n)},u_2^{\sigma(n)})\in\mathcal V(M_n)$.
Then $M_n\subset M_{n+1}$ and Corollary \ref{amalgam2} implies that
$M_{n+1}$ is a II$_1$ factor. Thus, $M$ defined as follows is a II$_1$ factor: $$M:=({\cup_{n\in\mathbb N}M_n})''.$$ 
\end{definition}

{\bf Convention.} For the rest of this section, $(M_n)_{n\in\mathbb N}$
and $M$ denote the II$_1$ factors introduced in Definition \ref{Mtilde}.

\begin{definition}\label{cofinal}
An ultrafilter $\mathcal U$ on a set $I$ is called {\it countably cofinal} if there exists a sequence $\{A_n\}_{n\in\mathbb N}\subset\mathcal U$ with $\cap_nA_n=\emptyset$.
\end{definition}

\begin{proposition}\label{twounitaries2} 
Let $u_1,u_2\in\mathcal U(M^{\mathcal U})$  such that $u_1^2=u_2^3=1$ and $\{u_1\}''\perp\{u_2\}''$, where $\mathcal U$ is a countably cofinal ultrafilter on a set $I$. 


Then there exist Haar unitaries $v_1,v_2\in M^{\mathcal U}$ such that $[u_1,v_1]=[u_2,v_2]=[v_1,v_2]=0$.

\end{proposition}

\begin{proof} 
Let $p,q_1,q_2,q_3\in M^{\mathcal U}$ be projections such that $u_1=2p-1$ and $u_2=q_1+\zeta q_2+\zeta^2 q_3$, where $\zeta=\exp(\frac{2\pi i}{3})$. We may clearly assume that $u_1\not=\pm 1$, so that $p\not=0,1$.


Since $M=(\cup_{n\in\mathbb N}M_n)''$ and $\mathcal U$ is cofinal, \cite[Lemma 2.2]{BCI15} gives  that $p,q_1,q_2,q_3\in\prod_{n\in\mathcal U}M_{k_n}$, for some  $(k_n)_{n\in I}\subset\mathbb N$. Moreover, the proof of \cite[Lemma 2.2]{BCI15} provides a function $f:I\rightarrow\mathbb N$ such that $\lim\limits_{n\rightarrow\mathcal U}f(n)=+\infty$.

Since $\{p\}''\perp \{q_1,q_2,q_3\}''$, by Lemma \ref{lift}, we can represent $p=(p_n)$ and $q_i=(q_{i,n})$, where $p_n,q_{n,i}\in M_{k_n}$ are projections such that $q_{1,n}+q_{2,n}+q_{3,n}=1$ and $\{p_n\}''\perp\{q_{1,n},q_{2,n},q_{3,n}\}''$, for every $n\in I$.
Let $u_{1,n}=2p_n-1$ and $u_{2,n}=q_{1,n}+\zeta q_{2,n}+\zeta^2 q_{3,n}$. Then 
 $u_1=(u_{1,n})$ and $u_2=(u_{2,n})$. Since $\{u_{1,n}\}''=\{p_n\}''\perp\{q_{1,n},q_{2,n},q_{3,n}\}''=\{u_{2,n}\}''$, we get $(u_{1,n},u_{2,n})\in\mathcal V(M_{k_n})$, for every $n\in I$.

 Since $\{(u_1^{k_n,j},u_2^{k_n,j})\}_{j\in\mathbb N}$ is dense in  $\mathcal V(M_{k_n})$, we can find $j_n\in\mathbb N$ such that \begin{equation}\label{approx}
\text{ $\|u_{1,n}-u_1^{k_n,j_n}\|_2+\|u_{2,n}-u_2^{k_n,j_n}\|_2\leq\frac{1}{f(n)}$, for every $n\in I$.   } 
\end{equation}
For $n\in I$, let $l_n\in\mathbb N$ such that $\sigma(l_n)=(k_n,j_n)$. Then $M_{\sigma(l_n)+1}=\Phi(M_{\sigma(l_n)},u_1^{k_n,j_n},u_2^{k_n,j_n})$.
Thus, by Corollary \ref{amalgam2}, we can find Haar unitaries $v_{1,n},v_{2,n}\in\mathcal U(M_{\sigma(l_n)+1})\subset\mathcal U(M)$ such that
\begin{equation}\label{comm}\text{$ 
[u_1^{k_n,l_n},v_{1,n}]=[u_2^{k_n,l_n},v_{2,n}]=[v_{1,n},v_{2,n}]=0$, for every $n\in I$.}    
\end{equation}
Finally, let $v_1=(v_{1,n}),v_2=(v_{2,n})\in\mathcal U(M^{\mathcal U}).$ Then $v_1,v_2$ are Haar unitaries and as $\lim\limits_{n\rightarrow\mathcal U}
f(n)=+\infty$, \eqref{approx} and \eqref{comm} together imply that  $[u_1,v_1]=[u_2,v_2]=[v_1,v_2]=0$. \end{proof}





In order to prove Theorem \ref{main lifting theorem}, we also need to find instances which guarantee that $M$ is full. This happens if $M_1$ has property (T):

\begin{proposition}\label{nonGamma}
Assume that $M_1$ has property (T). Then $M$ does not have property Gamma.
\end{proposition}

\begin{proof} Let $n\in\mathbb N$. Recall that $M_{n+1}=\varphi(M_n,u_1^{\sigma(n)},u_2^{\sigma(n)})$ and $M_1\subset M_n$. Since $M_1$ is a II$_1$ factor, we have that 
$M_1\nprec_{M_n}\{u_1^{\sigma(n)}\}''$ and $M_1\nprec_{M_n}\{u_2^{\sigma(n)}\}''$. By applying Corollary \ref{amalgam2} we derive that $M_1'\cap M_{n+1}=M_1'\cap M_n$. 
Thus, we get that $M_1'\cap M_n=\mathcal Z(M_1)=\mathbb C1$.
Since this holds for every $n\in\mathbb N$, we deduce that $M_1'\cap M=\mathbb C1$. Finally, since $M_1$ has property (T), by Proposition \ref{T fact}, we have that $M_1'\cap M^{\mathcal U}=(M_1'\cap  M)^{\mathcal U}=\mathbb C1$, where $\mathcal U$ is a free ultrafilter on $\mathbb N$. Hence, $ M'\cap M^{\mathcal U}=\mathbb C1$ and so $M$ does not have property Gamma. 
\end{proof}

\begin{proof}[Proof of Theorem \ref{main lifting theorem}]
Let $M_1$ be a II$_1$ factor with property (T), e.g., take $M_1=\text{L}(\text{SL}_3(\mathbb Z))$. Let $M$ be constructed as in Definition \ref{Mtilde}. 
The conclusion follows from Propositions \ref{twounitaries2} and \ref{nonGamma}. 
\end{proof}

\subsection{Proof of Theorem \ref{h=0} and its corollaries}
In this subsection, we prove that the II$_1$ factor $M$ from Theorem \ref{main lifting theorem} also satisfies the conclusion of Theorems \ref{h=0} and  \ref{corollary about non ee of free group factors with our thing} and Corollary \ref{Jekel corollary}.
To this end, we first show the following:

\begin{corollary}\label{23projections}
 Let $p,q_1,q_2,q_3\in M^{\mathcal U}$ be projections such that  $q_1+q_2+q_3=1$, where $\mathcal U$ is a countably cofinal ultrafilter on a set $I$. Assume that $\{p\}''\perp\{q_1,q_2,q_3\}''$.
 
 Then $h(\{p,q_1,q_2,q_3\}'',M^{\mathcal U})\leq 0$.\end{corollary}

\begin{proof}
Define $u_1,u_2\in \mathcal U(M^{\mathcal U})$ by $u_1=2p-1$ and $u_2=q_1+\zeta q_2+\zeta^2 q_3$, where $\zeta=\exp(\frac{2\pi i}{3})$. Then $u_1^2=u_2^3=1$ and $\{u_1,u_2\}''=\{p,q_1,q_2,q_3\}''$. Thus, by combining Lemma \ref{twounitaries} and Proposition \ref{twounitaries2} we get that $h(\{p,q_1,q_2,q_3\}'':M)=h(\{u_1,u_2\}'':M)\leq 0.$
\end{proof}

To prove that $h(M^{\mathcal U})\leq 0$, we will need an additional lemma:
\begin{lemma}\label{excision}
 Let $(A,\tau)$ be a diffuse tracial von Neumann algebra and $x\in A$ such that $x=x^*$ and $\tau(x)=0$. Let $\mathcal F$ be the set of projections $p\in A$ such that $\tau(xp)=0$. Then $\mathcal F''=A$.
\end{lemma}
 
\begin{proof}
We first prove the conclusion under the assumption that $A$ is abelian.  Let $x=y-z$ be the decomposition of $x$ into positive and negative parts. Let $q$ and $r$ be the support projections of $y$ and  $z$, respectively.  Since $0=\tau(x)=\tau(y)-\tau(z)$, we get that $\tau(y)=\tau(z)$.

Let  $e\in Aq$ be a projection. Since $A$ is diffuse and $\tau(ye)\leq\tau(y)=\tau(z)=\tau(zr)$, we can find a projection $f\in Ar$ such that $\tau(zf)=\tau(ye)$. Then we have that $e-f\in\mathcal F$.  Since $ef=0$, we get that $e+f=(e-f)^2\in\mathcal F''$ and thus $e\in\mathcal F''$, for every projection $e\in Aq$. Thus, $Aq\subset\mathcal F''$. Similarly, we conclude that $Ar\subset\mathcal F''$. Since $x(1-q-r)=0$, we also have that $A(1-q-r)\subset\mathcal F''$.
Since $A$ is abelian, it follows that $A\subset\mathcal F''$ and thus $\mathcal F''=A$.

For general $A$, let $B\subset A$ be a diffuse abelian von Neumann subalgebra. 
Note that $\tau(E_B(x))=0$ and that if $p\in B$ is a projection with $\tau(E_B(x)p)=0$, then $\tau(xp)=\tau(E_B(x)p)=0$ and so $p\in\mathcal F$. 
By applying the above proof to $B$ and $E_B(x)\in B$, we conclude that $B\subset\mathcal F''$. Since this holds for every diffuse abelian von Neumann subalgebra $B\subset A$, we conclude that $\mathcal F''=A$.
\end{proof}

\begin{theorem}\label{hleq0}
$h(M^{\mathcal U})\leq 0$, for any ultrafilter $\mathcal U$ on any set $I$.
\end{theorem}

\begin{proof}
If $\mathcal U$ is not countably cofinal, then $M^{\mathcal U}=M$ by \cite[Lemma 2.3]{BCI15}. Thus, if $\mathcal V$ is a free ultrafilter on $\mathbb N$, then Facts \ref{fact 1} and \ref{in the presence of the ultra} give that $h(M^{\mathcal U})=h(M)=h(M:M^{\mathcal V})\leq h(M^{\mathcal V})$.
This implies that in order to prove the conclusion, we may assume that $\mathcal U$ is countably cofinal. 

Assume that $\mathcal U$ is a countable cofinal ultrafilter and denote $P=M^{\mathcal U}$.
Since $P$ is a II$_1$ factor, we can find a unital, trace-preserving embedding of $S:=\text{L}(\mathbb Z/2\mathbb Z*\mathbb Z/2\mathbb Z)$ into $P$. Let $p,q\in S$ be two projections with $\tau(p)=\tau(q)=\frac{1}{2}$ which generate the two canonical copies of $\text{L}(\mathbb Z/2\mathbb Z)$ inside $S$.

Then  $\|q(2p-1)q\|_2=\sqrt{\tau(q(2p-1)q(2p-1))}=\frac{1}{2}$ and similarly $\|(1-q)(2p-1)(1-q)\|_2=\frac{1}{2}$. Let $x=(1-q)(2p-1)(1-q)\in (1-q)P(1-q)$. Then $x=x^*, \tau(x)=0$ and $x\not=0$. We define $\mathcal F$ to be the set of projections $r\in (1-q)P(1-q)$ such that $\tau(xr)=0$.

For $r\in\mathcal F$ we define $S_r:=\{p,q,r,1-q-r\}''$. Then $\tau((2p-1)q)=0$, $\tau((2p-1)r)=\tau(xr)=0$  and $\tau((2p-1)(1-q-r))=\tau(2p-1)-\tau((2p-1)q)-\tau((2p-1)r)=0$. Thus, $\{p\}''\perp \{q,r,1-q-r\}''$.
Altogether, we can apply Corollary \ref{23projections} to deduce that 
\begin{equation}\label{Sr}
    \text{$h(S_r:P)\leq 0$, for every $r\in\mathcal F$.}
\end{equation}
Since $S\subset S_r$, for all $r\in \mathcal F$, and $S$ is diffuse, combining Facts \eqref{joins} and \eqref{unions} with \eqref{Sr} we get that
\begin{equation}\label{Sr2}
h(\bigvee_{r\in\mathcal F}S_r:P)\leq 0.
\end{equation}
On the other hand, by Lemma \ref{excision} we have that $\mathcal F''=(1-q)P(1-q).$ This implies that
\begin{equation}\label{Sr3}
    \bigvee_{r\in\mathcal F}S_r=S\bigvee (1-q)P(1-q).
\end{equation}
Combining \eqref{Sr2} and \eqref{Sr3} we get $h(S\bigvee (1-q)P(1-q):P)\leq 0$. Similarly,  $h(S\bigvee qPq:P)\leq 0.$ Using again that $S$ is diffuse, Fact \ref{joins} implies that $h(S\bigvee qPq\bigvee (1-q)P(1-q):P)\leq 0$.
Since the projections $q$ and $1-q$ are equivalent in $S$, we get that $S\bigvee qPq\bigvee (1-q)P(1-q)=P$, which implies the desired conclusion that $h(P)=h(P:P)\leq 0$.
\end{proof}

Although this is not needed to derive our main results, we mention an easy consequence of the previous proof which seems of independent interest:

\begin{corollary}
 Let $M$ be a II$_1$ factor such that $h(M)>0$. Let $\Gamma=\mathbb Z/2\mathbb Z*\mathbb Z/3\mathbb Z$. Then there exists a homomorphism $\pi:\Gamma\rightarrow\mathcal U(M)$ such that $h(\pi(\Gamma)'':M)>0$.
\end{corollary}

\begin{proof}
 As $M$ is a II$_1$ factor, we can find a unital, trace-preserving embedding of $S:=\text{L}(\mathbb Z/2\mathbb Z*\mathbb Z/2\mathbb Z)$ into $M$. Let $p,q\in S$ be two projections with $\tau(p)=\tau(q)=\frac{1}{2}$ which generate the two canonical copies of $\text{L}(\mathbb Z/2\mathbb Z)$ inside $S$. Since $h(M)>0$ and $S\bigvee qMq\bigvee (1-q)M(1-q)=M$, Fact \ref{joins} gives that $h(S\bigvee qMq)>0$ or $h(S\bigvee (1-q)M(1-q))>0$. Assume, without loss of generality, that $h(S\bigvee qMq)>0$. 
 Given a projection $r\in qMq$, let $S_r=\{p,r,q-r,1-q\}''$.
 Since $S\bigvee qMq$ is generated by  $\{S_r\mid \text{$r\in qMq$ projection}\}$, Fact \ref{joins} implies that $h(S_r:M)>0$, for some projection $r\in qMq$. Since clearly $S_r=\pi(\Gamma)''$, for a homomorphism $\pi:\Gamma\rightarrow\mathcal U(M)$, the conclusion follows.
\end{proof}

\begin{proof}[Proof of Theorem \ref{h=0}]
Let $M_1$ be a II$_1$ factor with property (T), e.g., take $M=\text{L}(\text{SL}_3(\mathbb Z))$. Let $M$ be constructed as in Definition \ref{Mtilde}. By Theorem \ref{hleq0} and Proposition \ref{nonGamma} we get that $h(M^{\mathcal U})\leq 0$, for every ultrafilter $\mathcal U$, and $M$ does not have property Gamma.
\end{proof}


\begin{proof}[Proof of Theorem \ref{corollary about non ee of free group factors with our thing}]

Let $M$ be as in Theorem \ref{h=0}. Suppose that for some ultrafilters $\mathcal U,\mathcal V$ on sets $I,J$, there  exists an embedding of $M^{\mathcal U}$ into $N^{\mathcal V}$ that contains the diagonal inclusion of $N$.    By combining Theorem \ref{h=0} and Facts \ref{in the presence of the ultra} and \ref{fact 1} we get the following chain of inequalities:  $$0<h(N)=h(N: N^{\mathcal V})\leq h(N:M^{\mathcal U})\leq h(M^{\mathcal U}: M^{\mathcal U})=h(M^{\mathcal U})\leq0,$$

which is a contradiction. 
\end{proof}

\begin{proof}[Proof of Corollary \ref{Jekel corollary}]
Let $M$ be as in Theorem \ref{h=0}. 
For a sequence $(k_n)\subset\mathbb N$ and free ultrafilter $\mathcal U$ on $\mathbb N$ with $\lim\limits_{n\rightarrow\mathcal U}k_n=+\infty$, let  $\mathcal M=\prod_{\mathcal U} \mathbb M_{k_n}(\mathbb C)$. Then Fact \ref{Jekel fact} implies that $h(\mathcal M)>0$. By Theorem \ref{corollary about non ee of free group factors with our thing}, we deduce that $M$ is not elementarily equivalent to $\mathcal M$.
\end{proof}

The following remark was communicated to us separately by I. Goldbring and D. Jekel. 

\begin{remark}\label{GJ}
We give an explicit sentence distinguishing up to elementary equivalence any II$_1$ factor $M$ satisfying the properties of Theorem \ref{main lifting theorem} and any tracial von Neumann algebra $(N,\tau)$ with $h(N)>0$, in particular $\text{L}(\mathbb F_2)$. This follows readily from Lemma \ref{lift2}.
For unitaries $u_1,u_2,v_1,v_2\in M$, we define the formulae
    $$ \phi(u_1,u_2)=\| u_1^2 - 1 \|_2 + \| u_2^3 - 1 \|_2 + | \tau(u_1u_2) - \tau(u_1) \tau(u_2) | + | \tau(u_1u_2^2) - \tau(u_1) \tau(u_2^2) |$$
   $$ \psi(u_1,u_2,v_1,v_2) = \| u_1v_1-v_1u_1 
   \|_2 + \| u_2v_2-v_2u_2 \|_2 + \| v_1v_2-v_2v_1 
   \|_2 + \sum_{k \in \mathbb{Z} \setminus \{0\}} 2^{-k} ( |\tau(v_1^k)| + |\tau(v_2^k)|).$$
   
   Note that $\phi(u_1,u_2) = 0$ means that  $u_1^2=u_2^3=1$ and $\{u_1\}'' \perp \{u_2\}''$. We also note that D. Jekel observed that Lemma \ref{lift2} implies that the set  $\{u_1, u_2 \in\mathcal U(M)\mid \phi(u_1,u_2) = 0\}$ is a definable set over the theory of II$_1$ factors.

Theorem \ref{main lifting theorem} shows that  $M$ satisfies
   $\sup_{u_1,u_2\in \mathcal U(M),\phi(u_1,u_2)=0} \big(\inf_{v_1, v_2 \in \mathcal{U}(M)} \psi(u_1,u_2,v_1,v_2)\big) = 0$.
   In combination with Lemma \ref{lift2}, we derive  the existence of a function $\delta:[0,\infty)\rightarrow [0,\infty)$ such that $\delta(0)=0$, $\delta((0,\infty))\subset (0,\infty)$ and for all $\varepsilon>0$, the following implication holds for $u_1,u_2\in\mathcal U(M)$: if  $\phi(u_1,u_2)<\delta(\varepsilon)$, then $\phi'(u_1,u_2):=\inf_{v_1, v_2 \in \mathcal{U}(M)} \psi(u_1,u_2,v_1,v_2)<\varepsilon$.
   Moreover, $\delta$ is independent of the II$_1$ factor $M$, and can be taken to be continuous and strictly increasing.
   Then we have that $\delta(\phi'(u_1,u_2))\leq\phi(u_1,u_2)$, for every $u_1,u_2\in\mathcal U(M)$, and we can thus write the distinguishing sentence as follows:
   $$\sup_{u_1,u_2\in\mathcal U(M)}\max(0, \delta(\phi'(u_1,u_2))-\phi(u_1,u_2)).$$
   In fact, it is easy to see that a II$_1$ factor $M$ satisfies this sentence if and only if it satisfies the conclusion of Theorem \ref{main lifting theorem}.
   \end{remark}

\bibliographystyle{amsalpha}
\bibliography{innerAmen}



\end{document}